\documentclass[fleqn,final]{cas-sc}
\usepackage[numbers]{natbib}

\usepackage{amsfonts,amsmath,amssymb,amsthm,graphicx,float,hyperref,comment,color}
\usepackage{bm}
\usepackage[ruled, lined, linesnumbered]{algorithm2e}

\newtheorem{theorem}{Theorem}

\begin{document}
\let\WriteBookmarks\relax
\def\floatpagepagefraction{1}
\def\textpagefraction{.001}
\shorttitle{CPTP Low Rank Method}
\shortauthors{D. Appel\"{o} and Y. Cheng}

\title[mode = title]{Kraus is King: High-order Completely Positive and Trace Preserving (CPTP) Low Rank  Method for the  Lindblad Master Equation}


\author[1]{Daniel Appel\"{o}}[orcid=0000-0002-0378-4563]
\cormark[1]
\fnmark[1]
\ead{appelo@vt.edu}


\affiliation[1]{organization={Department of Mathematics, Virginia Polytechnic Institute and State University},
            addressline={225 Stanger Street}, 
            city={Blacksburg},
            postcode={24061}, 
            state={VA},
            country={USA}}

\author[1]{Yingda Cheng}[orcid=0000-0003-3229-2979]
\ead{yingda@vt.edu}

\cortext[cor1]{Corresponding author}
\fnmark[2]

\begin{abstract}
We design high order accurate methods that exploit low rank structure in the density matrix while respecting the essential structure of the Lindblad equation. Our methods preserves complete positivity and are trace preserving. 
\vspace{0.2cm}
%
\end{abstract}



\begin{keywords}
Completely positive and trace preserving, Kraus form  high order, integrating factor,  Lindblad master equation 
\end{keywords}

\maketitle

We are concerned with the evolution of a density matrix \mbox{$\rho(t) \in \mathbb{C}^{N \times N}$,} which is a unit-trace symmetric positive semi-definite (SPSD) matrix, describing the state of a quantum system (e.g. a quantum computer), where $\rho$ is governed by the Lindblad equation \cite{Lindblad76}  
\begin{align}\label{eqn::Lindblad}
   \frac{d}{dt} \rho(t) = -i(H \rho(t)-\rho(t) H) + \mathcal{L}_D \rho(t), \ \ \rho(0) = \rho_0.
    \end{align}
Here $H$ is the Hamiltonian, and $\mathcal{L}_D$ is on the form
    \begin{align}\label{eqn::LindbladDecay}
        \mathcal{L}_D \rho(t) = \sum_{\alpha} \gamma_\alpha \left(L_\alpha \rho(t) L_\alpha^
        \dag - \frac{1}{2}\left(L_\alpha^{\dag} L_\alpha \rho(t) + \rho(t)L_\alpha^{\dag} L_\alpha   \right)\right).
\end{align}
Here we exclusively consider time-independent $H$ and $L_\alpha$. 
For notational convenience, we consider the case with a single $\alpha$ and with $\gamma_\alpha = 1$, then we have  
\begin{equation}\label{eq:LindSmall}
        \frac{d}{dt} \rho(t)
         = -i(H \rho(t) - \rho(t) H) +  \left(L \rho(t) L^\dag - \frac{1}{2}\left(L^{\dag} L \rho(t) + \rho(t)L^\dag L   \right)\right).
\end{equation}
Our goal is to design high order accurate methods that exploit low rank structure in $\rho$ while respecting the essential structure of the Lindblad equation.  A defining feature of the Lindblad equation is that its evolution of $\rho$ preserves complete positivity (CP) and is trace preserving (TP) \cite{manzano2020short}. Choi's theorem (see e.g. \cite{manzano2020short}) states: {\em A linear map $\mathcal{G}$ is CP iff it can be represented as}
$$
\mathcal{G}\rho=\sum_l G_l^\dagger \rho G_l,
$$
and this provides a way to construct numerical methods that preserve CP. Let $\rho^n \approx \rho(t_n)$ be an approximation to the density matrix at time $t_n$ and let a numerical method be defined as a {\em Kraus map} $\mathcal{G}$ that takes $\rho^n$ to $\rho^{n+1}$, that is 
\[
\rho^{n+1} \equiv \mathcal{G}\rho^n=\sum_l G_l^\dagger \rho^n G_l.
\] 
Such a  method is obviously CP and preservation of trace can be achieved by normalization \mbox{$\rho^{n+1} \leftarrow \rho^{n+1}/{\rm Tr}(\rho^{n+1})$} at the end of each timestep. 

To design CPTP schemes is a non-trivial task. For example, it is known,  \cite{riesch2019analyzing}, that no explicit Runge-Kutta method is CP. Consequently, considerable effort has been spent on designing the numerical scheme that are on Kraus form. A pioneering work in this direction is \cite{Steinbach1995}, where a Taylor series expansion technique is used and explicit formulae defining schemes up to fourth order accuracy are provided. Later work  \cite{cao2021structure, wang2024simulation,ding2024simulating} are based on various techniques including the use of Duhamel's principle and principles of  stochastic unraveling. While these aforementioned schemes are of interest we remark that the techniques used when deriving them makes the construction of high order accurate schemes somewhat involved. 

Further, another property of the density matrix that is prominent in many quantum systems, in particular for the ones with low entropy, is that it can be effectively approximated by low rank techniques.  { The low rank property is expected to be relevant for ``systems that are weakly coupled to the environment, or for the early stage of the dynamics of systems initialized in a pure state'' \cite{gravina2024adaptive} (for further examples and motivation for low rankedness see  \cite{gravina2024adaptive})}. When present, low rank structure can and should be exploited to accelerate simulations and reduce memory footprint. Moreover, as shown later in this paper, the Kraus form plays well with low rank structures and operations can be performed directly on low rank Cholesky factors. In the literature, a standard approach for evolving low rank structures is the time-dependent variational principle  (TDVP) \cite{dirac1930note}. TDVP has been implemented in \cite{LeBris1,gravina2024adaptive} to approximate the time evolution of low rank density matrices. The main idea in TVDP is to evolve solution of an equation projected onto the low rank matrix manifold. However, since the tangent projection of the Lindbladian onto the low rank density matrix manifold in general may not be Lindbladian, a numerical scheme with the CP property resulting from TDVP is not possible to the best of the authors' knowledge. 

In response to these challenges we here propose a different, systematic and straightforward, approach to constructing high order CPTP low rank schemes for the Lindblad equation. Our approach uses the integrating factor (IF) method (or Lawson method \cite{lawson1967generalized}), and yields a class of high order integrators that will be CP as long as easily verifiable constraints on the entries of Butcher tableaus are met. Our schemes can be used in full or low rank form. In the low rank form we use so called low rank step truncation (also called ensemble truncation \cite{donatella2021continuous, mccaul2021fast} or eigenvalue truncation \cite{chen2021low} in the physics literature). This approach uses the truncated singular value decomposition (SVD) which, as we prove in this paper, is also a CP map. Therefore, combining the IF and step truncation,  we obtain a class of high order low rank CPTP schemes.

\subsection*{Integrating Factor Induced  CP  Scheme}
Following e.g. \cite{Steinbach1995} we rewrite (\ref{eq:LindSmall}) in terms of an effective Hamiltonian $J$ 
\begin{equation}
\label{eqn:eff_eqn}
        \frac{d}{dt} \rho(t)
         = \left(J\rho(t) + \rho(t)J^\dag \right) + L \rho(t) L^\dag \equiv \mathcal{L}_J \rho(t) +\mathcal{L}_L \rho(t),
\end{equation}
where 
$J = -i H_{\rm eff}, \   H_{\rm eff} = H + \frac{1}{2i} L^\dag L. 
$
The operator $\mathcal{L}_L $ is already on Kraus form, and hence the key to designing a CP scheme is the appropriate discretization of the operator $\mathcal{L}_J $. 

Our schemes start from rewriting (\ref{eqn:eff_eqn}), using integrating factor, as
\begin{equation} \label{eq:ode_rewrite} 
\frac{d}{dt} (e^{-\mathcal{L}_J t} \rho)=e^{-\mathcal{L}_J t} \mathcal{L}_L \rho.
\end{equation}
Note that the $e^{-\mathcal{L}_J t}$ acts on $\rho$ and $\mathcal{L}_L \rho$ which are both SPSD matrices. 

Before defining the scheme, for reasons soon to be revealed, we consider the action of $e^{\mathcal{L}_J \Delta t}$ onto a SPSD matrix $q(t)$. Such a matrix admits a Cholesky factorization $q(t)=V(t) V(t)^\dag$. After one timestep, at time $t + \Delta t$, we have  $q(t+\Delta t)=e^{\mathcal{L}_J \Delta t} q(t)$ which can also be represented as $q(t+\Delta t)=V(t+\Delta t) V(t+\Delta t)^\dag.$ Here, as a consequence of the equivalence of 
\begin{equation} \label{eq:jeqn} \frac{d}{dt} q(t)
         =  Jq(t) + q(t)J^\dag,   
\end{equation} 
and 
\begin{equation}    \label{eq:V_ode} 
\frac{d}{dt} V(t) = JV(t),
\end{equation} 
we have that $V(t+\Delta t)=U(\Delta t) V(t)$, where $U(\Delta t)$ is the exact flow operator $e^{ J \Delta t}$. Note that $q(t+\Delta t)=U(\Delta t) q(t) U(\Delta t) ^\dag,$ which is on Kraus form and thus CP. Further note that even if an approximate flow $\hat{U}(\Delta t)$ is used (for example by solving (\ref{eq:V_ode}) by a numerical scheme) the update $q(t+\Delta t)=\hat{U}(\Delta t) q(t) \hat{U}(\Delta t) ^\dag$ will still be CP.  

Before we define our method we recall the definition of an $s$-stage, $k$-th order explicit RK method for the system of ordinary differential equations
\begin{equation} \label{eq:ode_generic} 
\frac{d y(t)}{d t} = f(t,y), 
\end{equation}
characterized by a Butcher Tableau $(\bm{\mathsf{A}},\bm{\mathsf{b}},\bm{\mathsf{c}})$ using a timestep $\Delta t$. Let $y_0$ be the current solution, then the approximate solution at the the next timestep is given by 
\begin{eqnarray*}  
y_1 = y_0 + \Delta t \sum_{i=1}^s b_i f(t_0 + c_i \Delta t,y^{(i)})&&y^{(i)} = y_0 + \Delta t \sum_{j=1}^{i-1} a_{ij} f(t_0 + c_j \Delta t,y^{(j)}), \, i = 1,\ldots ,s.
\end{eqnarray*}  
Applying the RK method to \eqref{eq:ode_rewrite}   yields
\begin{eqnarray*} 
 e^{-\mathcal{L}_J \Delta t}\rho_1 =   \rho_0 + \Delta t  \sum_{i = 1}^s  b_i e^{-\mathcal{L}_J c_i \Delta t} \mathcal{L}_L \rho^{(i)}& &e^{-c_i \Delta t \mathcal{L}_J} \rho^{(i)}  = \rho_0 +\Delta t \sum_{j=1}^{i-1} a_{ij} e^{-\mathcal{L}_J c_j \Delta t} \mathcal{L}_L \rho^{(j)},\ \  i = 1,\ldots, s. 
\end{eqnarray*}
Regrouping terms and using the   flow operator $U(\cdot)$,   the  IF method to \eqref{eqn:eff_eqn} is thus given by 
\begin{eqnarray} 
&& \rho^{(i)}  = U(c_i \Delta t)
 \rho_0 U(c_i \Delta t)^\dag
  + \Delta t \sum_{j=1}^{i-1} a_{ij} U((c_i-c_j) \Delta t)
   \mathcal{L}_L \rho^{(j)} U((c_i-c_j) \Delta t)^\dag,\ \  i = 1,\ldots ,s,  \notag \\
&&\rho_1 =  U(  \Delta t)\rho_0 U(  \Delta t)^\dag + \sum_{i = 1}^s  b_i \Delta t U((1-c_i) \Delta t) \mathcal{L}_L\rho^{(i)}U((1-c_i) \Delta t)^\dag,  \label{eq:RKCP1}
\end{eqnarray}
  which is on Kraus form and is thus CP as long as the elements of $\bm{\mathsf{b}}, \bm{\mathsf{A}}$ are non-negative (this is the case for example for the classic fourth order Runge-Kutta method and all strongly stability preserving methods). To fully specify the scheme,  the flow operator $U(\tau) V$ should be defined. Beyond   the exact flow computed via the matrix exponential, we can also use a $k$-th order Taylor series method  where derivatives $\frac{d^m V}{dt^m}, m=1, \ldots k$ are computed recursively from (\ref{eq:V_ode}). The rationale for using a Taylor series method is that the overall order of accuracy is decide by the base RK scheme, thus there is no need to approximate the flow more accurately. We note that it would also be possible to approximate the flow operator with an implicit method such as the trapezoidal method.   

\subsection*{Truncated SVD is a CP Map}
The scheme \eqref{eq:RKCP1} is amenable to low rank implementation based on a low rank
 truncation of a sum of $k$ low rank matrices $\mathbf{R}=R_1R_1^\dagger+\ldots R_kR_k^\dagger$, with $R_j \in \mathbb{C}^{N \times r_j}$. Suppose $W = [R_1,\ldots,R_k]$, then $\mathbf{R}=W W^\dagger$  and the low rank truncation $\mathcal{T}_{\epsilon,r_{\rm max}}[\mathbf{R}]$ is defined as the truncated singular value decomposition (SVD) of $\mathbf{R}$ according to the rank threshold $r_{\rm max}$ and energy cutoff $\epsilon.$ In practice, this can be efficiently implemented by the following procedure. Compute the pivoted QR factorization $QR =  W \Pi$, then compute the SVD of the small matrix $R \Pi = \hat{U} \Sigma \hat{V}^\dagger$. Then $\mathcal{T}_{\epsilon,r_{\rm max}}[\mathbf{R}] =\hat{W} \hat{W}^\dag,$ where $\hat{W}=Q \hat{U}(:,1:r)\Sigma(1:r,1:r)$, $r= \min(r_{\epsilon},r_{\rm max})$, and $r_{\epsilon}$ is the smallest integer such that $ \sum_{j = r_{\epsilon}+1}^N  \sigma_{j}^2   \le \epsilon^2$.

Importantly, in the following theorem, we demonstrate that  the truncated SVD is a  CP map.
\begin{theorem} 
The truncated SVD operator $\mathcal{T}_{\epsilon,r_{\rm max}}[A],$ where $A$ is SPSD,  is on Kraus form, and thus is a CP map.
\end{theorem}
\begin{proof} 
Since $A$ is SPSD, its SVD has the form $A=U\Lambda U^\dagger,$ and 
$ 
\mathcal{T}_{\epsilon,r_{\rm max}}[A]=U D  \Lambda D^\dagger U^\dagger
$ where 
\mbox{$D= {\rm diag}(\underbrace{1, \ldots 1,}_r 0 \ldots, 0).$}
Therefore, 
$
\mathcal{T}_{\epsilon,r_{\rm max}}[A]=U D U^\dagger A (U D U^\dagger)^\dagger
$ is on Kraus form.
\end{proof}

\subsection*{CPTP Low Rank IF Scheme}
We are now ready to present the low rank version of the IF scheme \eqref{eq:RKCP1}. Starting from $V_0 \in \mathbb{C}^{N\times r}$, for stage $i$ out of $s$ form the tall and skinny matrix $W^{(i)}$ with columns  
\[
U(c_i \Delta t)V_0, \text{ and, when } i>1, \text{ columns } \sqrt{\Delta t a_{i,j}} U((c_i-c_j) \Delta t) \, L V^{(j)}, \ \ j = 1,\ldots,i-1.
\]
Then perform the low rank truncation on $W^{(i)}$, $\mathcal{T}_{\epsilon,r_{\rm max}}[W^{(i)}]$ to obtain $V^{(i)}.$ The solution at the next timestep is found by forming the matrix $W$ whose columns are  
\[
U(\Delta t)V_0, \  \  \  \  \sqrt{b_i\Delta t} U((1-c_i) \Delta t) L V^{(i)}, \ \ i = 1,\ldots,s. 
\]
followed by the truncation 
\[
V_1V_1^\dagger = \mathcal{T}_{\epsilon, r_{\rm max}}[W W^\dag].
\]
Finally, with the re-normalization \mbox{$\rho_1 \leftarrow \rho_1/{\rm Tr}(\rho_1)$}, this defines a class of high-order low rank CPTP schemes. The truncation can be used with a fixed user prescribed $\epsilon$ and $r_{\rm max} = \infty$, or with a user prescribed $r_{\rm max }$ and $\epsilon = 0$. In the first case a suitable choice is $\epsilon \sim \Delta t^{k+1}$ when a $k-$th order Runge-Kutta methods is used for (\ref{eq:RKCP1}). 

We emphasize the following features of the method. For the low rank scheme we only need to store the low rank factors, e.g. $V_0, V^{(i)},$ and never the full density matrix. This means that if the rank factor is low, we will achieve significantly reduced memory footprint and computational cost. We also emphasize that the trace renormalization will preserve the order of accuracy of the original discretization.

\subsection*{Alternative Truncation}
We note that when forming $W^{(i)}$ and $W$ one can carry out additional truncation with a more aggressive $\epsilon \sim \Delta t^{k}$ on the columns that have a $\mathcal{O}(\sqrt{\Delta t})$ pre-factor, prior to performing the truncation above. While the added truncation steps increases the cost for forming $W^{(i)}$ and $W$  it can be that, problems where $L$ has many terms, the additional cost is offset by the lower cost due to the smaller number of columns in $W^{(i)}$ and $W$. 
 
\section*{Numerical Examples}
We now present numerical examples illustrating the features of the new schemes. We denote by IF the scheme based on (\ref{eq:RKCP1}) with the flow operator $U(\cdot)$ approximated by matrix exponentiation. We denote by IF-LR and IF-LR-T the low rank implementation of the scheme based on (\ref{eq:RKCP1}) with the flow operator $U(\cdot)$ approximated by matrix exponentiation and the flow operator $U(\cdot)$ approximated by a Taylor series expansion, respectively. For all the results in this section we use the Butcher Tableau in (\ref{eq:RKCP1}) corresponding to the classic fourth order accurate Runge-Kutta method. {Note that we prefer to use non-dimensionalized quantities but that in order to repeat the numerical experiment in \cite{riesch2019analyzing} we adopt their notation, which includes physical units.}

\subsection*{Confirmation of CPTP}
In \cite{riesch2019analyzing} Riesch and Jirauschek studied CP for numerical methods applied to the Liouville-von Neumann equation. Here we repeat their numerical experiment (see, Section 4 in \cite{riesch2019analyzing}) with identical Hamiltonian and initial data, but add decoherence modeled by the two Lindblad operators
\[
L_1 = \frac{1}{\sqrt{2}}10^{5} a,\ \  L_2 = 10^{5} a^\dagger a. 
\]  
Here $a$ is the lowering matrix with elements $a_{l,l+1} = \sqrt{l}, l = 1,\ldots,5$.

In the left part of Figure \ref{fig:loss_CP} we display $\rho_{33}(t)$ computed using the IF method and the classic fourth order Runge-Kutta method, both using a timestep $\Delta t = 0.1 fs$ {(the same as in \cite{riesch2019analyzing})}. As can be seen in the inset the Runge-Kutta method does not stay CPTP while the IF  method does. In the right part of Figure \ref{fig:loss_CP} we again display $\rho_{33}(t)$ obtained with the same methods but now with a five times larger timestep. For this larger timestep, not only is the loss of CPTP is more pronounced (especially at early times), but what is also clear is that the IF method is much more accurate for this example, despite both methods being of order four. In the right subfigure in Figure \ref{fig:loss_CP} we have also included results using IF-LR-T using a sixth order accurate Taylor series, yielding results that are very similar.

\begin{figure}[htb]
\begin{center}
\includegraphics[width=0.46\textwidth,trim={0.0cm 0.0cm 0.0cm 0.0cm},clip]{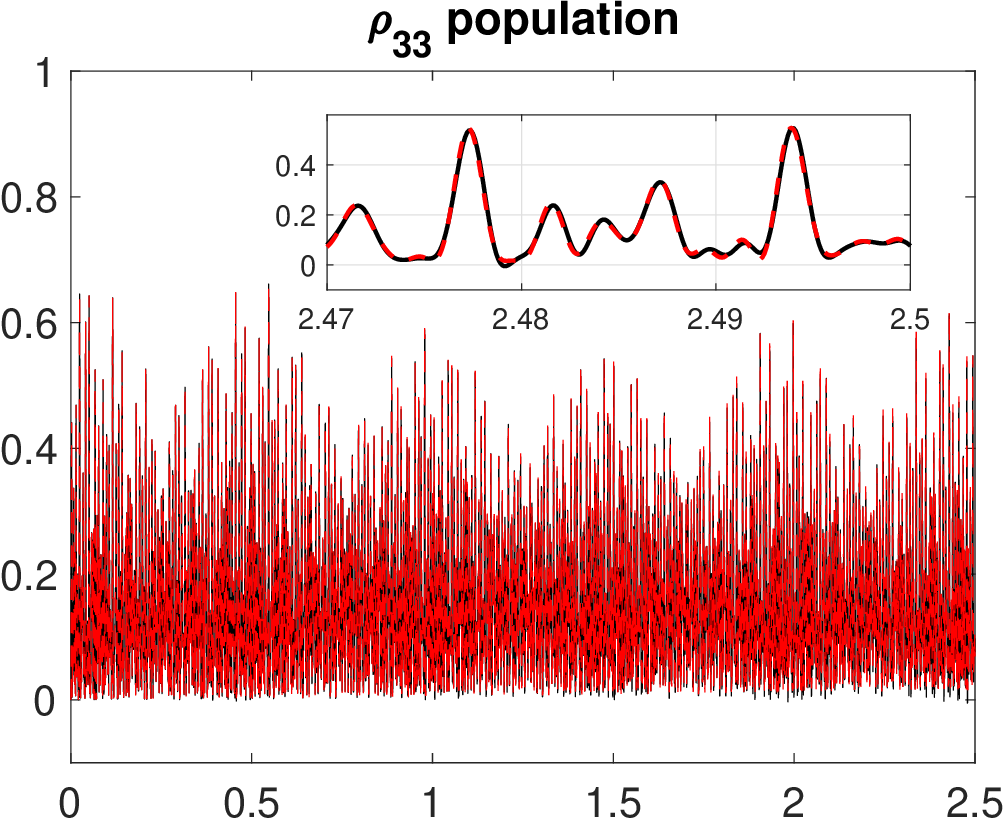}
\includegraphics[width=0.46\textwidth,trim={0.0cm 0.0cm 0.0cm 0.0cm},clip]{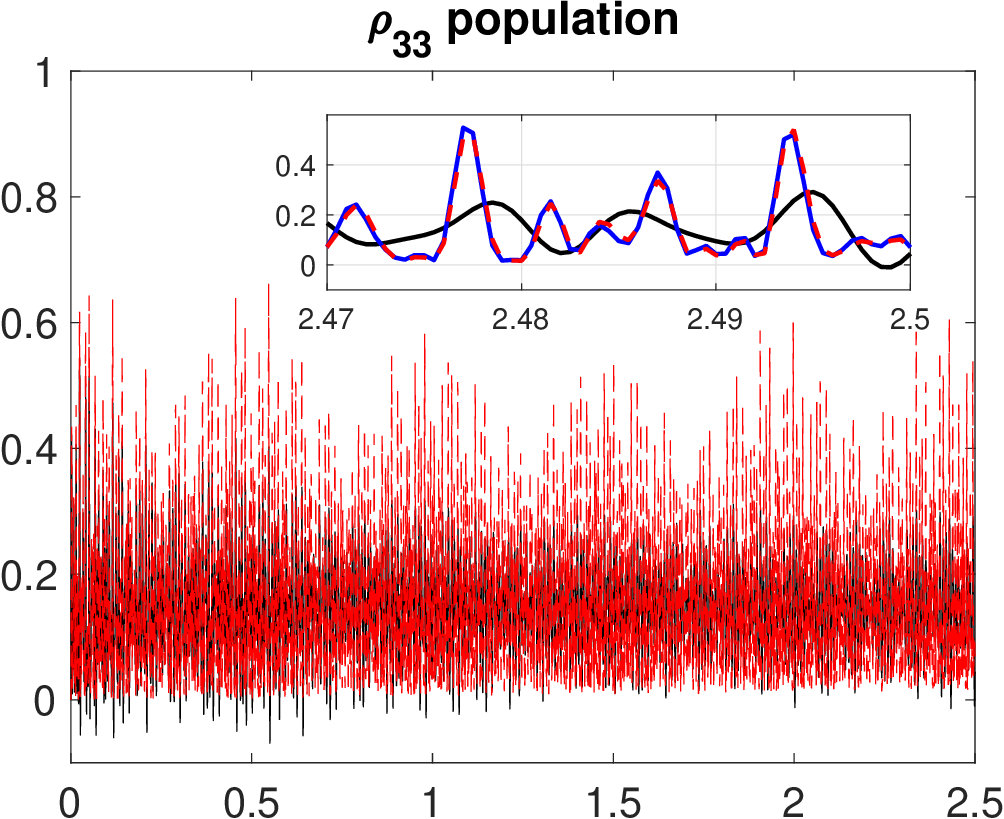}
\caption{Evolution of the population $\rho_{33}$ in the example from \cite{riesch2019analyzing}. {The units for the $x$-axis is pico seconds. The number of timesteps in the left and right insets are 300 and 60 respectively.} To the left we compare the classic RK4 method (black) and the integrating factor method with matrix exponentiation for the flow (red, dashed) using a timestep of 0.1 femto-seconds. To the right we use a five times larger timestep and have also included results when the flow is computed using a Taylor series method of order 6. \label{fig:loss_CP}}
\end{center}
\end{figure}

\subsection*{Oscillation Revival in the Jaynes-Cumming Model}
\begin{figure}[htb]
\begin{center}
\includegraphics[width=0.32\textwidth,trim={0.0cm 0.0cm 0.0cm 0.0cm},clip]{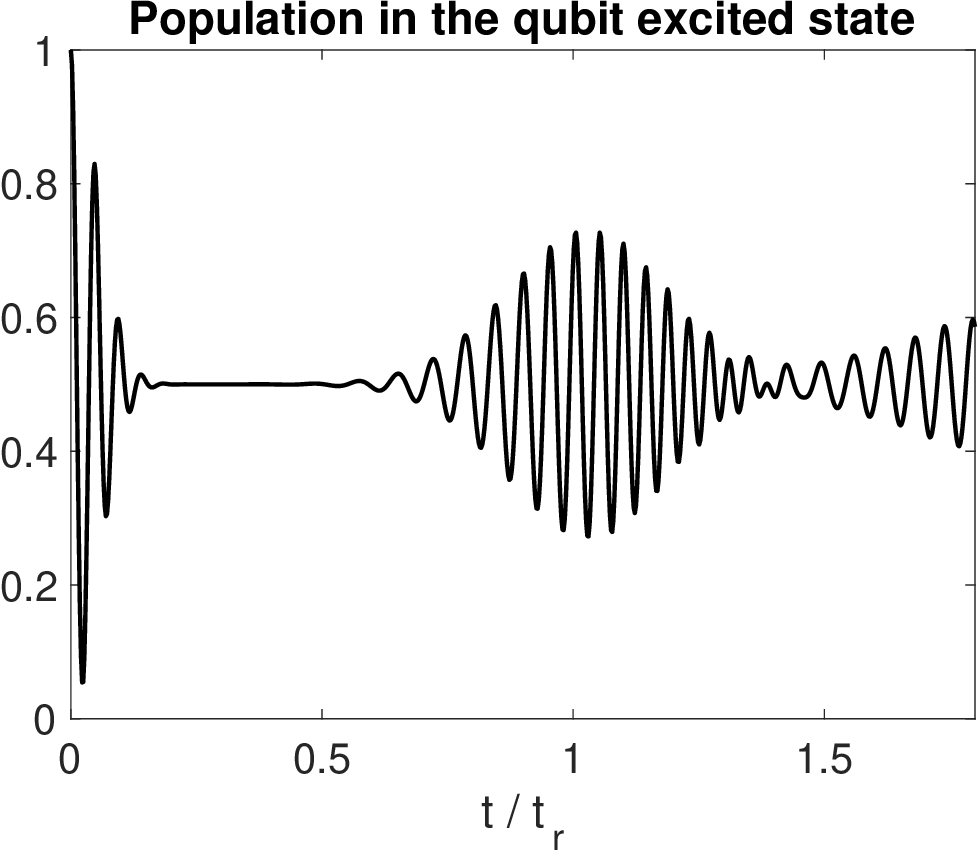}
\includegraphics[width=0.33\textwidth,trim={0.0cm 0.0cm 0.0cm 0.0cm},clip]{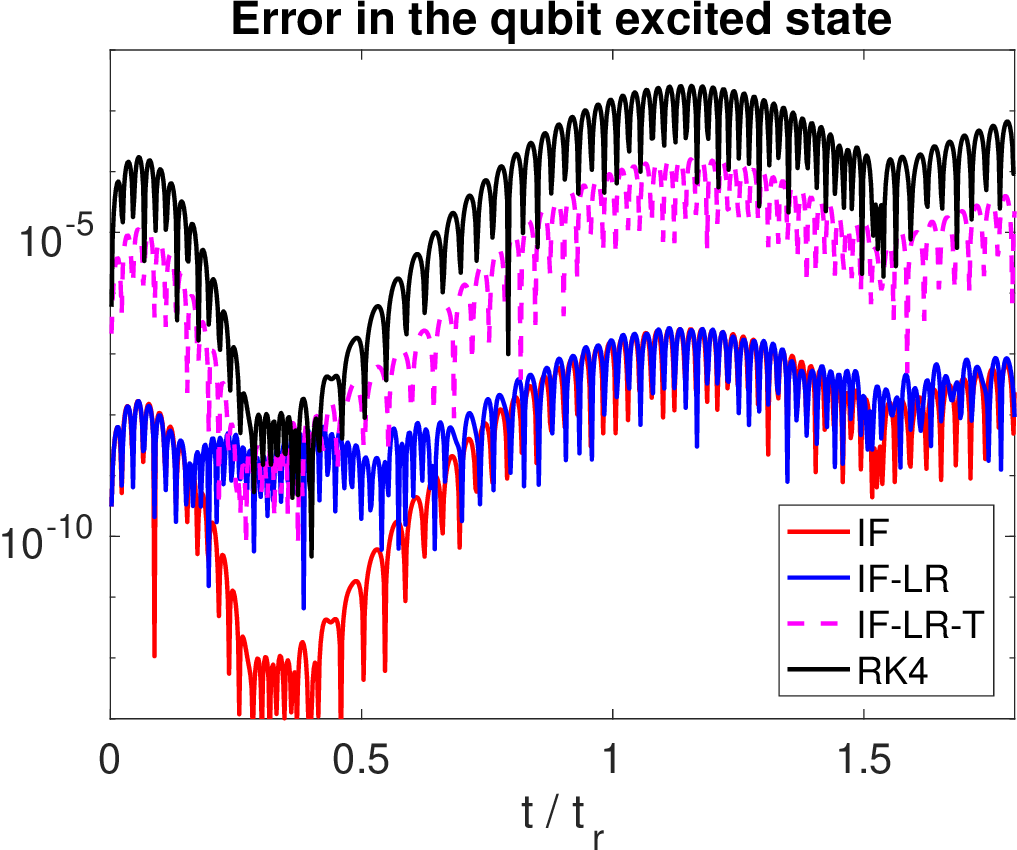}
\caption{Computation of revival for a small $m$ case using the classic fourth order accurate Runge-Kutta method, the full-rank integrating factor method and two low-rank methods (matrix exponentiation and Taylor series). See the text for explanation of the subfigures. \label{fig:revive_1}}
\end{center}
\end{figure}
In this example we consider Jaynes-Cumming model of a two level atom (qubit) interacting with a single quantized field mode (cavity) with $m$ energy levels. Under the rotating wave transformation and at exact resonance the (interaction) Hamiltonian takes the particularly simple form, 
\cite{PhysRevA.44.5913}, 
\[
H = \lambda ( b \sigma^+ + b^{\dagger} \sigma^-),
\]
and there is only one Lindblad operator $L = \sqrt{\kappa} b$. Here 
\[
b = I_{2 \times 2} \otimes \hat{b}, \ \  \sigma^+ = \begin{pmatrix}
0 & 0\\
1 & 0
\end{pmatrix} \otimes I_{m \times m}, 
\ \  \sigma^- = \begin{pmatrix}
0 & 1\\
0 & 0
\end{pmatrix} \otimes I_{m \times m}
\]
and $\hat{b}$ is the $m \times m$ lowering matrix with elements $\hat{b}_{l,l+1} = \sqrt{l}, l = 1,\ldots,m-1$. At the initial time the atom is in the excited state and the cavity is in the coherent state 
\[
{\bf v} \sim \sum_{n=0}^{m-1} \frac{|v|^n}{\sqrt{n!}} {\bf e}_n,   
\]
that is, the initial density matrix is $\rho = VV^\dagger$, with 
\[
V =  \begin{pmatrix}
0 \\
1
\end{pmatrix} \otimes \frac{{\bf v}}{\|{\bf v} \|}.
\]
Here we choose $v = \sqrt{m/3}$ so that the last terms in the sum for ${\bf v}$ are small also for moderate $m$. With these choices and for large $m$ the populations in the atom collapses to $1/2$ but revives at $t_r = 2 \pi |v| / \lambda$, see \cite{PhysRevA.44.5913}.   

 \begin{table}[] 
 \begin{center} 
\begin{tabular}{|c|c|c|c|c|c|c|c|c|c|c|c|c|} 
\hline
steps & IF & & RK & &  ME-7 && T-7&& ME-9 && T-9 &\\
\hline
\hline
 200 &  1.1(-4) &       &  3.6(-1) &       &  1.1(-4) &   &  6.1(-2) &       &  1.1(-4) &       &  6.1(-2)    & \\ 
 400 &  6.8(-6) & 4.0 &  6.3(-2) & 2.5 &  9.1(-6) &  3.6   &  4.1(-3) & 3.9 &  6.8(-6) & 4.0 &  4.1(-3) & 3.9  \\ 
 800 &  4.2(-7) & 4.0 &  4.2(-3) & 3.9 &  1.2(-5) & -0.35 &  2.6(-4) & 4.0 &  4.4(-7) & 3.9 &  2.6(-4) & 4.0  \\ 
 \hline
 \end{tabular}
 \caption{Here ME-X and T-X refers to the low rank method using matrix exponentiation and a fourth order accurate Taylor series method respectively. The -X refers to using $\epsilon = 10^{-X}$. IF refers to the method defined by (\ref{eq:RKCP1}). RK refers to the classic fourth order accurate Runge-Kutta method. Steps refers to the number of timesteps. The numbers below each method is the $L_2$ errors in time in the excited state in the qubit and the numbers to the right are the estimated orders of convergence. \label{tab:revive_1}} 
 \end{center} 
 \end{table} 
As a first example we take $m = 30$ and $\kappa = 0.001$ and simulate until time $1.8t_r$. The time evolution of the population in the excited state is displayed in Figure \ref{fig:revive_1}. To verify the order of accuracy we compare solutions using the classic fourth order Runge-Kutta method (RK4) directly for (\ref{eqn::Lindblad}) and together with the IF method. We also use the IF-LR and IF-LR-T (with a fourth order Taylor series) methods. In Table \ref{tab:revive_1} we report $L_2$ (in time) errors for the population of the excited state along with observed rates of convergence. For the low-rank schemes we use two different tolerances (see the caption of Table \ref{tab:revive_1}). We see the expected rate of convergence for all cases with the exception of the IF-LR method at the looser tolerance. The reason for this can be understood from the right subfigure in Figure \ref{fig:revive_1}, where we have displayed the errors for the 800 timestep computation with the tighter tolerance. As can be seen the methods that use matrix exponentiation are significantly more accurate. This is to be expected for problems where the Hamiltonian dominates the Lindbladian. As the error due to the timestepping is quite small for the IF-LR method, the looser tolerance represents the dominant source of error and no convergence with decreasing $\Delta t$ is observed. Our observation is that the results in Figure \ref{fig:revive_1} are typical, with the traditional RK method being least accurate and the full-rank IF method being the most accurate. The accuracy of the, considerably more efficient, low-rank methods are in between.  

Instead of choosing the tolerance as a fixed number it can be chosen as $\epsilon = \Delta t^q$. In Figure  \ref{fig:revive_1_tol_conv} we display the errors in the population of the excited state in the qubit for $q = 2,3,4,5$ and for the low-rank methods using matrix exponentiation and Taylor series for the flow. As can be seen the rate of convergence is $q-1$, showing that $\epsilon$ should be chosen to be of the same size as the local truncation error of the underlying Runge-Kutta method.    

\begin{figure}[htb]
\begin{center}
\includegraphics[width=0.40\textwidth,trim={0.0cm 0.0cm 0.0cm 0.0cm},clip]{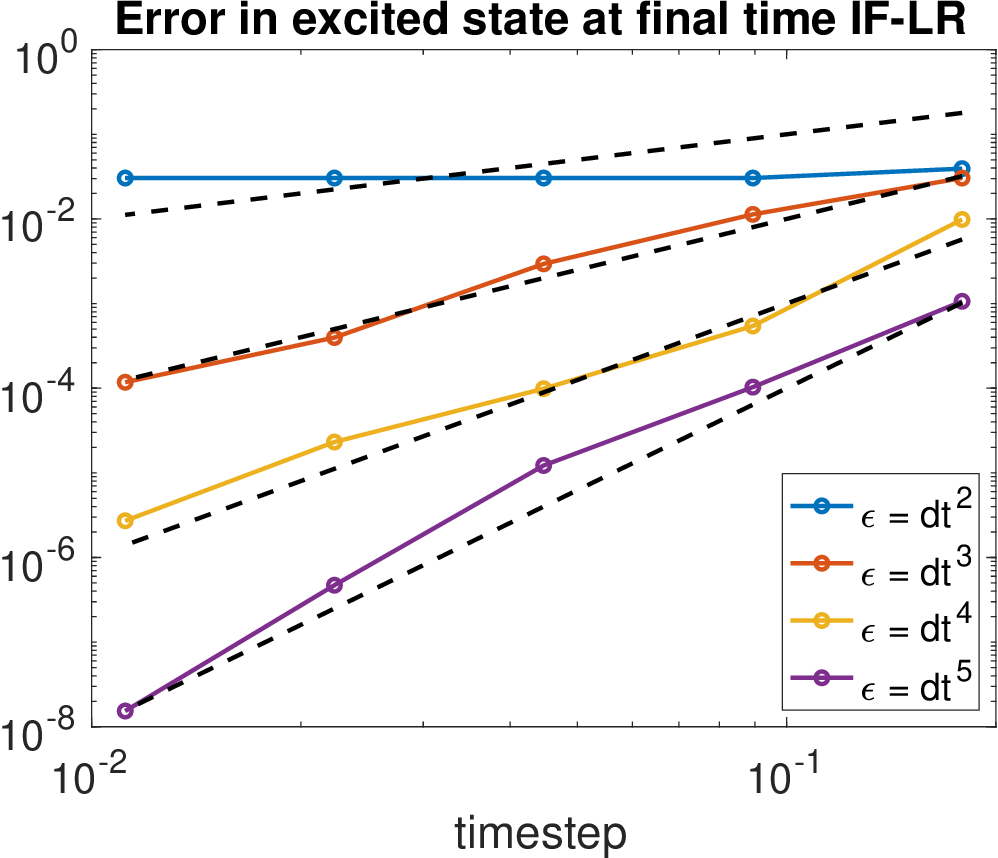} \ \ \ \ 
\includegraphics[width=0.40\textwidth,trim={0.0cm 0.0cm 0.0cm 0.0cm},clip]{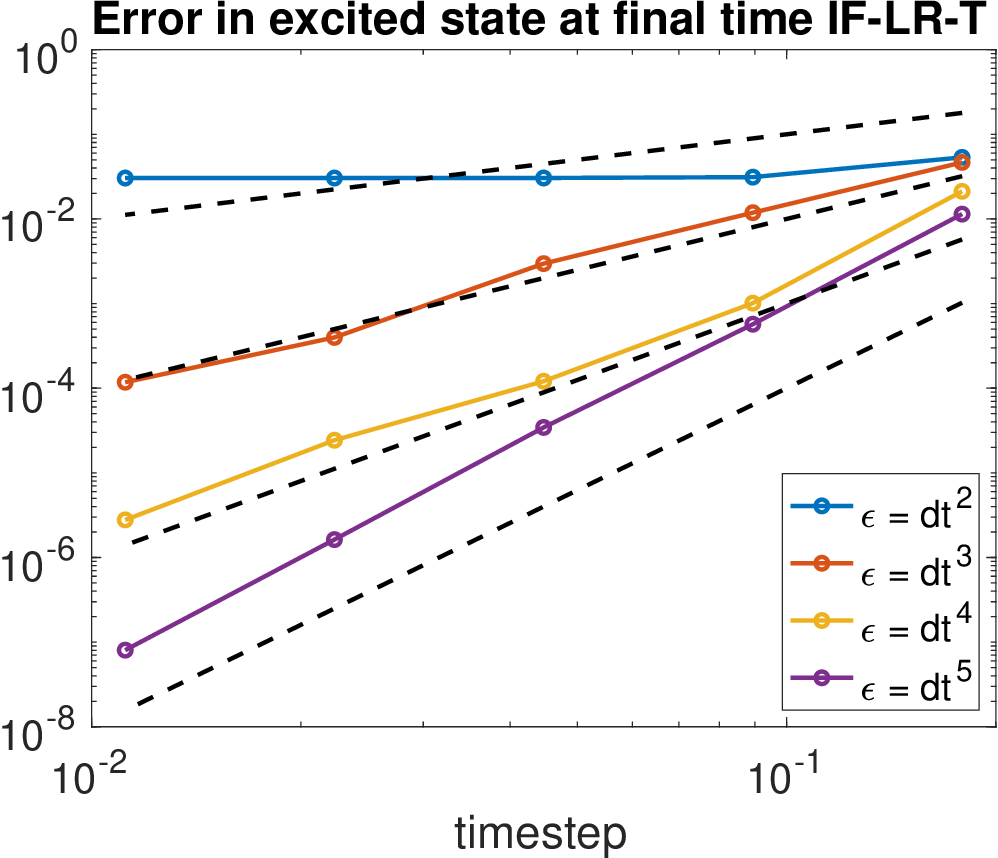}
\caption{Errors at the final time for the two low-rank methods (matrix exponentiation (left) and Taylor series (right)) as a function of the timestep. The dashed lines are orders, one to four. See the text for explanation of the subfigures. \label{fig:revive_1_tol_conv}}
\end{center}
\end{figure}

In this example we take $m = 150$, $\kappa = 0.002/9$ and simulate until time $3t_r$ using 4000 timesteps. We present results for $\epsilon = 10^{-3}, 10^{-5}, 10^{-7}$ using the low-rank method with RK4 and 4th a order Taylor series for the flow. The results are displayed in Figure \ref{fig:revive_2}. To the left we have plotted the population for the excited state (the insets are zoom-ins) for $\epsilon = 10^{-7}$. The second subfigure displays the population for the excited state for each or the tolerances (vertically offset by 1 and 2 for clarity). The computations with the largest $\epsilon$ over-predicts the amplitude of the revivals. The third figure displays the rank parameter as a function of time and for the different tolerances. As can be seen the rank parameter grows slowly even though the solution is highly oscillatory. The  rank for the computation with the largest $\epsilon$ is one for the duration of the simulation, which can be an explanation to the over-prediction of the revival. The rightmost figure displays the difference between the solutions obtained with $\epsilon = 10^{-3}, 10^{-5}$ compared to $\epsilon =  10^{-7}$. While the errors are slightly larger than $\epsilon$ it does appear that the reduction is around 100 times in-between the two cases.     

\begin{figure}[htb]
\begin{center}
\includegraphics[width=0.24\textwidth,trim={0.0cm 0.0cm 0.0cm 0.0cm},clip]{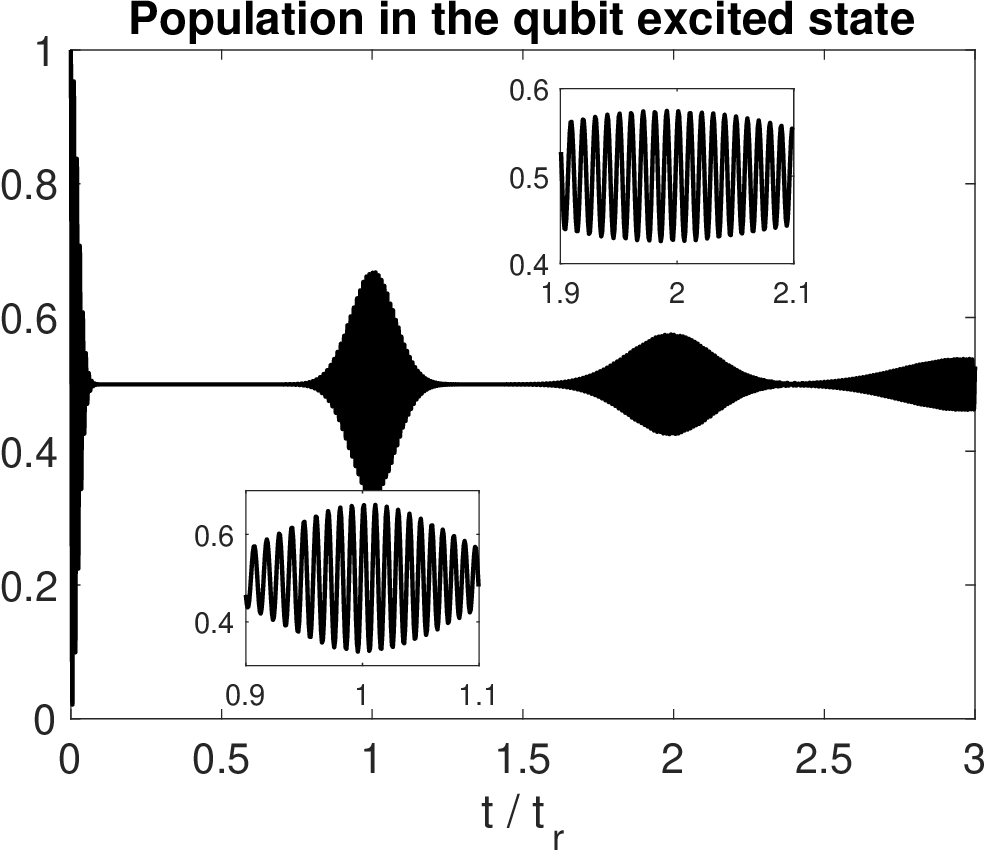}
\includegraphics[width=0.24\textwidth,trim={0.0cm 0.0cm 0.0cm 0.0cm},clip]{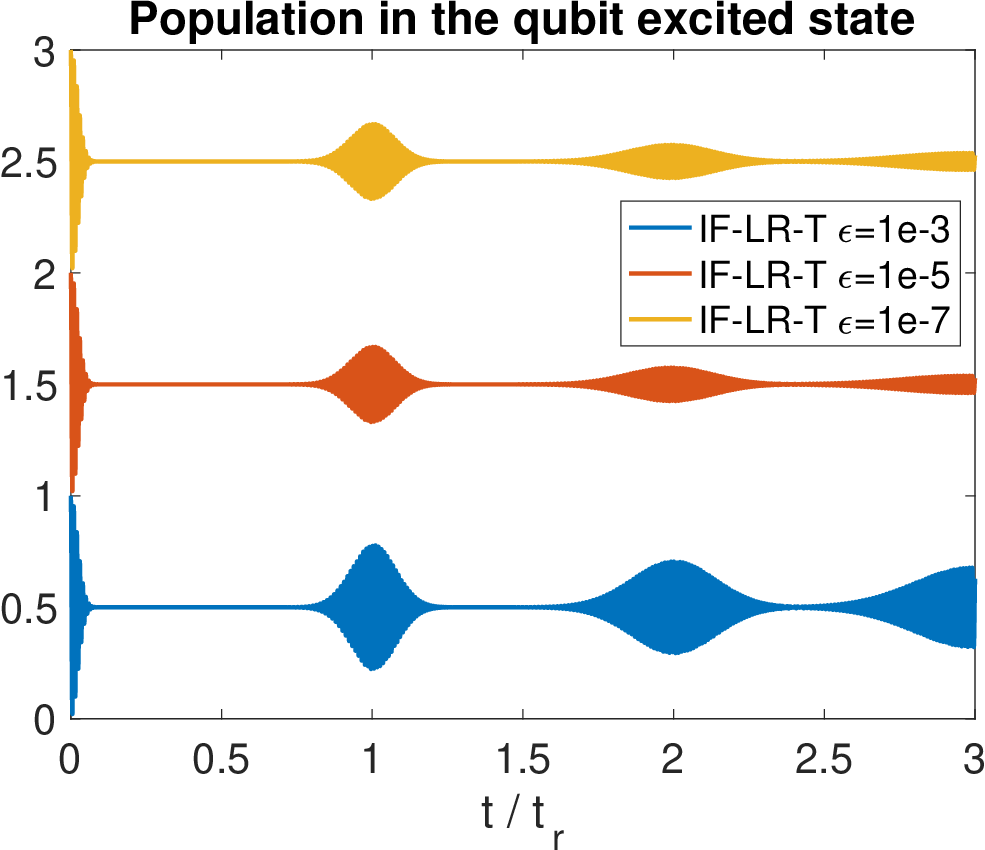}
\includegraphics[width=0.24\textwidth,trim={0.0cm 0.0cm 0.0cm 0.0cm},clip]{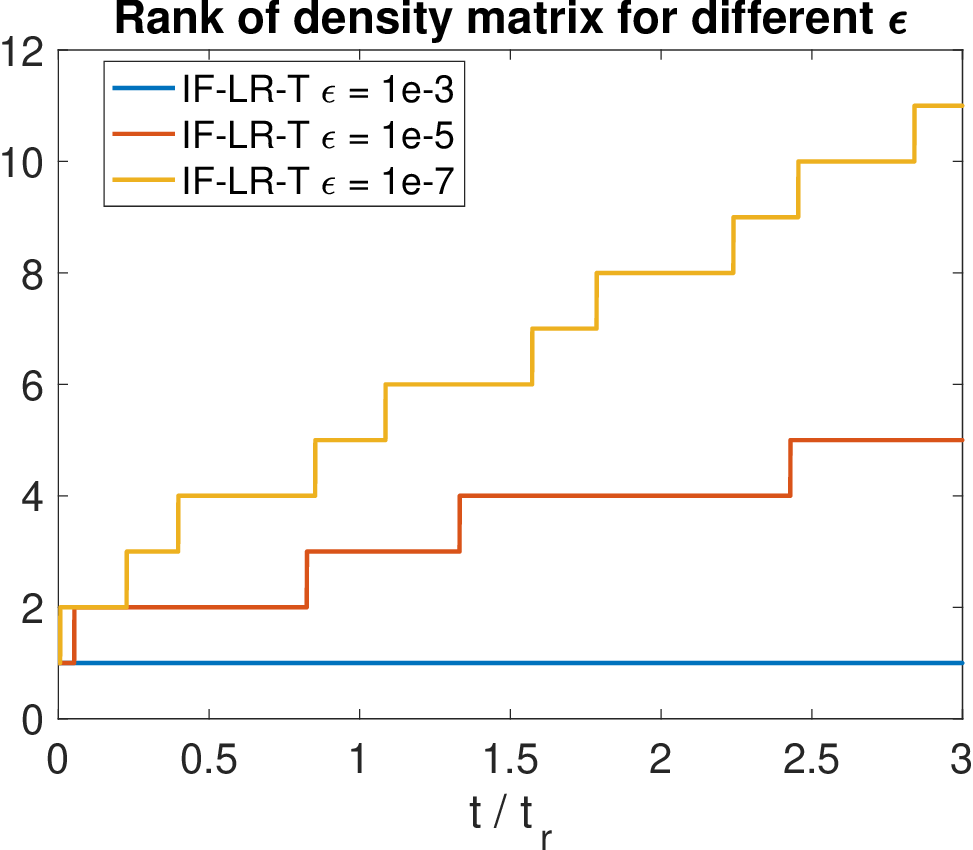}
\includegraphics[width=0.25\textwidth,trim={0.0cm 0.0cm 0.0cm 0.0cm},clip]{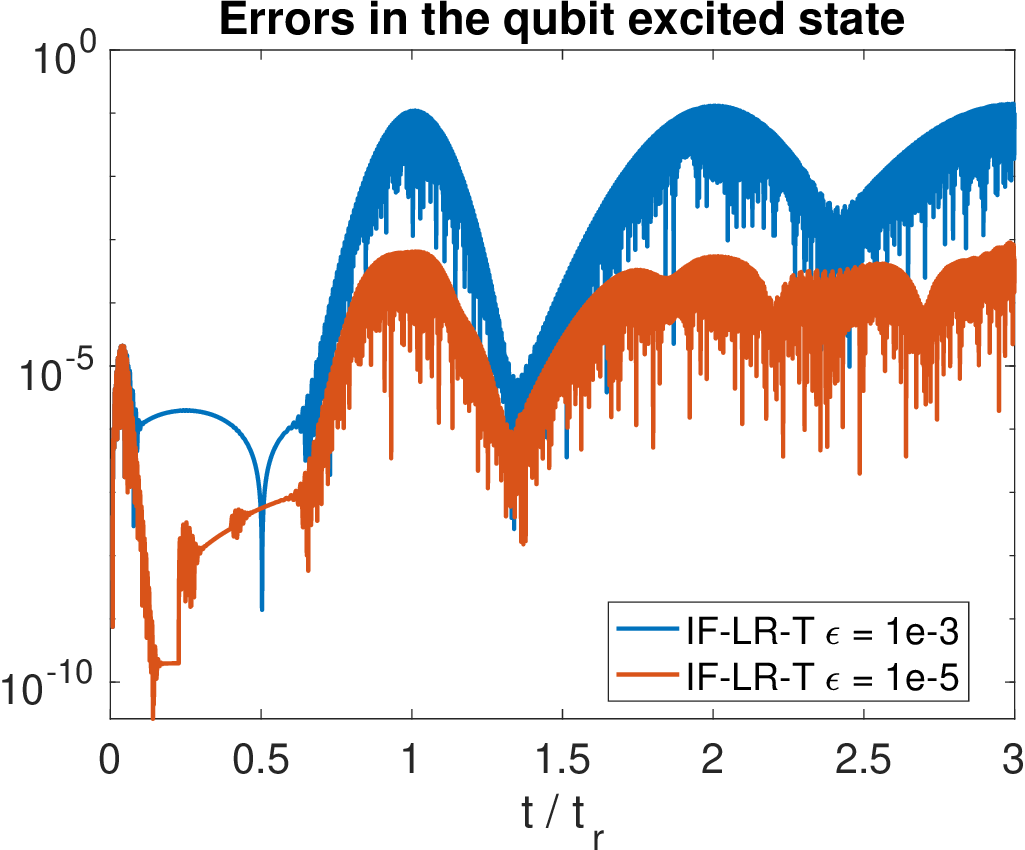}
\caption{Low-rank computation of revival for a large $m$ case. See the text for explanation of the subfigures. \label{fig:revive_2}}
\end{center}
\end{figure}


In summary, this paper develops a class of high order low rank CPTP schemes for  the Lindblad equation with time-independent Hamiltonian and jump operators. We demonstrate in numerical examples that the schemes are accurate and efficient for problems with low rank properties. In subsequent work, we will consider the general case when the Hamiltonian and jump operators depend on time. We also plan to develop tensor network based low rank schemes \cite{werner2016positive,sulz2024numerical} to tackle the challenging case of many qubits. { As pointed out by one of the reviewers, the low rank techniques used here can likely be adopted to other methods, for example those in \cite{Steinbach1995,cao2021structure}.}

\subsection*{Acknowledgment and Disclaimer}
This material is based upon work supported by the U.S. Department of Energy, Office of Science, Advanced Scientific Computing Research (ASCR), under Award Number DE-SC0025424.

This report was prepared as an account of work sponsored by an agency of the United States Government. Neither the United States Government nor any agency thereof, nor any of their employees, makes any warranty, express or implied, or assumes any legal liability or responsibility for the accuracy, completeness, or usefulness of any information, apparatus, product, or process disclosed, or represents that its use would not infringe privately owned rights. Reference herein to any specific commercial product, process, or service by trade name, trademark, manufacturer, or otherwise does not necessarily constitute or imply its endorsement, recommendation, or favoring by the United States Government or any agency thereof. The views and opinions of authors expressed herein do not necessarily state or reflect those of the United States Government or any agency thereof.

This material is based upon work supported by the National Science Foundation under Grant Number 2436319. Any opinions, findings, and conclusions or recommendations expressed in this material are those of the author(s) and do not necessarily reflect the views of the National Science Foundation.

\bibliographystyle{elsarticle-num} 
\bibliography{refs}

\begin{thebibliography}{10}
\expandafter\ifx\csname url\endcsname\relax
  \def\url#1{\texttt{#1}}\fi
\expandafter\ifx\csname urlprefix\endcsname\relax\def\urlprefix{URL }\fi
\expandafter\ifx\csname href\endcsname\relax
  \def\href#1#2{#2} \def\path#1{#1}\fi

\bibitem{Lindblad76}
G.~Lindblad, On the generators of quantum dynamical semigroups, Communications
  in Mathematical Physics 48~(2) (1976) 119--130.

\bibitem{manzano2020short}
D.~Manzano, A short introduction to the {L}indblad master equation, Aip
  Advances 10~(2) (2020) 025106.

\bibitem{riesch2019analyzing}
M.~Riesch, C.~Jirauschek, Analyzing the positivity preservation of numerical
  methods for the {L}iouville-von {N}eumann equation, Journal of Computational
  Physics 390 (2019) 290--296.

\bibitem{Steinbach1995}
J.~Steinbach, B.~M. Garraway, P.~L. Knight, High-order unraveling of master
  equations for dissipative evolution, Phys. Rev. A 51 (1995) 3302--3308.

\bibitem{cao2021structure}
Y.~Cao, J.~Lu, Structure-preserving numerical schemes for {L}indblad equations,
  arXiv preprint arXiv:2103.01194 (2021).

\bibitem{wang2024simulation}
K.~Wang, X.~Li, Simulation-assisted learning of open quantum systems, Quantum 8
  (2024) 1407.

\bibitem{ding2024simulating}
Z.~Ding, X.~Li, L.~Lin, Simulating open quantum systems using {H}amiltonian
  simulations, PRX Quantum 5~(2) (2024) 020332.

\bibitem{gravina2024adaptive}
L.~Gravina, V.~Savona, Adaptive variational low-rank dynamics for open quantum
  systems, Physical Review Research 6~(2) (2024) 023072.

\bibitem{dirac1930note}
P.~A. Dirac, Note on exchange phenomena in the thomas atom, in: Mathematical
  proceedings of the Cambridge philosophical society, Vol.~26, Cambridge
  University Press, 1930, pp. 376--385.

\bibitem{LeBris1}
C.~Le~Bris, P.~Rouchon, Low-rank numerical approximations for high-dimensional
  {L}indblad equations, Phys. Rev. A 87 (2013) 022125.

\bibitem{lawson1967generalized}
J.~D. Lawson, Generalized runge-kutta processes for stable systems with large
  lipschitz constants, SIAM Journal on Numerical Analysis 4~(3) (1967)
  372--380.

\bibitem{donatella2021continuous}
K.~Donatella, Z.~Denis, A.~Le~Boit{\'e}, C.~Ciuti, Continuous-time dynamics and
  error scaling of noisy highly entangling quantum circuits, Physical Review A
  104~(6) (2021) 062407.

\bibitem{mccaul2021fast}
G.~McCaul, K.~Jacobs, D.~I. Bondar, Fast computation of dissipative quantum
  systems with ensemble rank truncation, Physical Review Research 3~(1) (2021)
  013017.

\bibitem{chen2021low}
Y.-T. Chen, C.~Farquhar, R.~M. Parrish, Low-rank density-matrix evolution for
  noisy quantum circuits, NPJ Quantum Information 7~(1) (2021) 61.

\bibitem{PhysRevA.44.5913}
J.~Gea-Banacloche, Atom- and field-state evolution in the {J}aynes-{C}ummings
  model for large initial fields, Phys. Rev. A 44 (1991) 5913--5931.

\bibitem{werner2016positive}
A.~H. Werner, D.~Jaschke, P.~Silvi, M.~Kliesch, T.~Calarco, J.~Eisert,
  S.~Montangero, Positive tensor network approach for simulating open quantum
  many-body systems, Physical Review Letters 116~(23) (2016) 237201.

\bibitem{sulz2024numerical}
D.~Sulz, C.~Lubich, G.~Ceruti, I.~Lesanovsky, F.~Carollo, Numerical simulation
  of long-range open quantum many-body dynamics with tree tensor networks,
  Physical Review A 109~(2) (2024) 022420.

\end{thebibliography}
\end{document}